\documentclass[reqno]{amsart}
\usepackage{amssymb,amsmath,amsthm,xcolor,enumerate,hyperref,cleveref}

\newtheorem{thrm}{Theorem}[section]

\newtheorem{lem}[thrm]{Lemma}
\newtheorem{prop}[thrm]{Proposition}

\theoremstyle{definition}

\newtheorem{rem}[thrm]{Remark}

\crefrangeformat{equation}{#3(#1)#4--#5(#2)#6}

\crefname{thrm}{Theorem}{Theorems}
\crefname{lem}{Lemma}{Lemmas}
\crefname{cor}{Corollary}{Corollaries}
\crefname{prop}{Proposition}{Propositions}
\crefname{defn}{Definition}{Definitions}
\crefname{exm}{Example}{Examples}
\crefname{rem}{Remark}{Remarks}
\crefname{section}{Section}{Sections}
\crefname{equation}{\unskip}{\unskip}
\crefname{enumi}{\unskip}{\unskip}

\renewcommand{\iff}{\Leftrightarrow}

\newcommand{\vf}{\varphi}
\newcommand{\0}{\theta}
\newcommand{\m}{^{-1}}
\newcommand{\tl}{\tilde}

\newcommand{\red}[1]{\textcolor{red}{#1}}

\begin{document}

\title[Jordan Isomorphisms of Finitary Incidence Algebras]{Jordan Isomorphisms\\ of Finitary Incidence Algebras}

\author{Rosali Brusamarello}
\address{Departamento de Matem\'atica, Universidade Estadual de Maring\'a, Maring\'a --- PR, CEP: 87020--900, Brazil}
\email{brusama@uem.br}

\author{\'Erica Z. Fornaroli}
\address{Departamento de Matem\'atica, Universidade Estadual de Maring\'a, Maring\'a --- PR, CEP: 87020--900, Brazil}
\email{ezancanella@uem.br}

\author{Mykola Khrypchenko}
\address{Departamento de Matem\'atica, Universidade Federal de Santa Catarina,  Campus Reitor Jo\~ao David Ferreira Lima, Florian\'opolis --- SC, CEP: 88040--900, Brazil}
\email{nskhripchenko@gmail.com}

\begin{abstract}
Let $X$ be a partially ordered set, $R$ a commutative $2$-torsionfree unital ring and $FI(X,R)$ the finitary incidence algebra of $X$ over $R$. In this note we prove that each $R$-linear Jordan isomorphism of $FI(X,R)$ onto an $R$-algebra $A$ is the near-sum of a homomorphism and an anti-homomorphism.
\end{abstract}

\subjclass[2010]{Primary 16S50, 17C50; Secondary 16W10}

\keywords{Jordan isomorphism, near-sum, homomorphism, anti-homomorphism, finitary incidence algebra}

\maketitle

\section*{Introduction}\label{intro}

The study of Jordan maps between rings originated in 1940's in the works by G.~Ancochea~\cite{Ancochea42,Ancochea47}, I.~Kaplansky~\cite{Kaplansky47} and L.~K.~Hua~\cite{Hua49}. In 1950's the topic was further developed by N.~Jacobson and C.~E.~Rickart in \cite{Jacobson-Rickart50}, by I.~N.~Herstein in \cite{Herstein56} and by M.~F.~Smiley in~\cite{Smiley57}. The description of Jordan homomorphisms turned out to be closely related to the description of homomorphisms and anti-homomorphisms. In particular, Jacobson and Rickart proved in~\cite[Theorem 7]{Jacobson-Rickart50} that, for $n\ge 2$, every Jordan homomorphism of the ring $M_n(R)$ of all $n\times n$-matrices over a ring $R$ into an arbitrary ring is the sum of a homomorphism and an anti-homomorphism. The classical theorem by Herstein says that every Jordan homomorphism of a ring onto a prime ring of characteristic different from $2$ and $3$ is either a homomorphism or an anti-homomorphism (see~\cite[Theorem H]{Herstein56}). Smiley improved in~\cite{
Smiley57} Herstein's result by eliminating the restriction that the characteristic be different from $3$. For further generalizations see the works of W.~E.~Baxter and W.~S.~Martindale III~\cite{Baxter-Martindale79}, M.~Bre\v{s}ar~\cite{Bresar89,Bresar91}, W.~S.~Martindale III~\cite{Martindale67,Martindale90}, K.~McCrimmon~\cite{McCrimmon89}.

L.~Moln\'ar and P.~\v{S}emrl~\cite{Molnar-Semrl98} initiated the investigation of Jordan maps on the ring $T_n(\mathcal C)$ of upper triangular matrices. They proved in~\cite[Corollary 4]{Molnar-Semrl98} that each Jordan automorphism of $T_n(\mathcal C)$ is either an automorphism, or an anti-automorphism, where $\mathcal C$ is a field with at least $3$ elements. K.~I.~Beidar, M.~Bre\v{s}ar and M.~A.~Che\-bo\-tar generalized this result in ~\cite{BeBreChe} by showing that each Jordan isomorphism of $T_n(\mathcal C)$ onto a $\mathcal C$-algebra is either an isomorphism, or an anti-isomorphism, provided that $\mathcal C$ is a $2$-torsionfree unital commutative ring without non-trivial idempotents and $n\ge 2$. D.~Benkovi\v{c} introduced in~\cite{Benkovic05} the notion of a near-sum, which helped him to describe in~\cite[Theorem 4.1]{Benkovic05} all Jordan homomorphisms $T_n(\mathcal C)\to A$, where $\mathcal C$ is an arbitrary $2$-torsionfree commutative ring and $A$ is a $\mathcal C$-algebra. E.~Akkurt, M.~Akkurt 
and G.~P.~Barker~\cite{Akkurts-Barker} extended Benkovi\v{c}'s result to structural matrix algebras $T_n(\mathcal C,\rho)$, where $\rho$ is either a partial order, or a quasi-order each of whose equivalence classes contains at least $2$ elements (see~\cite[Theorem 2.1]{Akkurts-Barker}). 

Observe that $T_n(\mathcal C,\rho)$ is isomorphic to the incidence algebra of the ordered set $(\{1, \ldots, n\},\rho)$ over the ring $\mathcal C$. In this paper we partially generalize the result by E.~Akkurt {\it et al.} to the case of finitary incidence algebras $FI(X,R)$, namely, we show that each $R$-linear Jordan isomorphism of $FI(X,R)$ onto an $R$-algebra $A$ is the near-sum of a homomorphism and an anti-homomorphism, where $X$ is an arbitrary partially ordered set and $R$ a commutative $2$-torsionfree unital ring.

The article is organized as follows. After giving all the necessary background information on Jordan homomorphisms and (finitary) incidence algebras in \cref{prelim}, we proceed with some technical lemmas in \cref{jiso-FI} which allow us to prove in \cref{vf|_D-homo} that each Jordan isomorphism $\vf:FI(X,R)\to A$ restricts to a homomorphism (and an anti-homomorphism) of the commutative subalgebra $D(X,R)$ consisting of the so-called diagonal elements. Furthermore, restricting $\vf$ to the subalgebra $\tilde I(X,R)$ generated by the matrix units \cref{e_xy}, one may directly apply the argument from the proof of~\cite[Theorem 2.1]{Akkurts-Barker} to decompose $\vf|_{\tilde I(X,R)}$ into the near sum $\psi+\theta$, where $\psi$ is a homomorphism $\tilde I(X,R)\to A$ and $\theta$ is an anti-homomorphism $\tilde I(X,R)\to A$. We show in \cref{psi-homo} that $\psi$ extends to a homomorphism $\tilde\psi:FI(X,R)\to A$, and an analogous fact is proved for $\theta$ in \cref{identities-for-theta} skipping some details.
 Our main result is \cref{near-sum}, which says that $\tilde\psi+\tilde\theta$ is a decomposition of $\vf$ into a near-sum.

\section{Preliminaries}\label{prelim}

\subsection{Jordan homomorphisms}\label{jordan-homo}
Let $A$ and $B$ be algebras over a commutative ring $R$. An $R$-linear map $\varphi:A\to B$ is called a \emph{Jordan homomorphism}, if it satisfies
\begin{align}
 \vf(a^2)&=\vf(a)^2,\label{vf(a^2)}\\
  \vf(aba)&=\vf(a)\vf(b)\vf(a),\label{vf(aba)}
\end{align}
for all $a,b\in A$. A bijective Jordan homomorphism is called a \emph{Jordan isomorphism}.

Each homomorphism, as well as an anti-homomorphism, is a Jordan homomorphism. The sum of a homomorphism $\psi:A\to B$ and an anti-homomorphism $\0:A\to B$ is a Jordan homomorphism, if $\psi(a)\0(b)=\0(a)\psi(b)=0$ for all $a,b\in A$. A more general construction was introduced by D.~Benkovi\v{c} in \cite{Benkovic05}. Suppose that $A$ can be represented as the direct sum of $R$-subspaces $A_0\oplus A_1$, where $A_0$ is a subalgebra of $A$ and $A_1$ is an ideal of $A$. Let $\psi:A\to B$ be a homomorphism and $\0:A\to B$ an anti-homomorphism, such that $\psi|_{A_0}=\0|_{A_0}$ and $\psi(a)\0(b)=\0(a)\psi(b)=0$ for all $a,b\in A_1$. Then the \emph{near-sum} of $\psi$ and $\0$ (with respect to $A_0$ and $A_1$) is the $R$-linear map $\vf:A\to B$, which satisfies $\vf|_{A_0}=\psi|_{A_0}=\0|_{A_0}$ and  $\vf|_{A_1}=\psi|_{A_1}+\0|_{A_1}$. One can show that $\vf$ is a Jordan homomorphism in this case.

Notice that the substitution of $a$ by $a+c$ in \cref{vf(aba)} gives
\begin{align}\label{vf(abc+cba)}
 \vf(abc+cba)=\vf(a)\vf(b)\vf(c)+\vf(c)\vf(b)\vf(a).
\end{align}
We shall also use the following fact (see Corollary 2 of \cite[Theorem 1]{Jacobson-Rickart50}). If $e$ is an idempotent, such that $ea=ae$, then
\begin{align}\label{vf(e)vf(a)=vf(a)vf(e)}
 \vf(a)\vf(e)=\vf(e)\vf(a)=\vf(ae).
\end{align}
In particular, if $A$ has identity $1$, then $\vf(1)$ is the identity of $\vf(A)$. Another particular case of \cref{vf(e)vf(a)=vf(a)vf(e)}: if $ea=ae=0$, then
\begin{align}\label{vf(e)vf(a)=0}
 \vf(e)\vf(a)=\vf(a)\vf(e)=0.
\end{align}

\subsection{Finitary incidence algebras}\label{fin-inc-alg}
Let $(X,\le)$ be a partially ordered set and $R$ a commutative ring with identity. The {\it incidence space $I(X,R)$ of $X$ over $R$} is defined to be the set of functions $f: X\times X\to R$, such that $f(x,y)=0$ if $x\not\le y$, with the natural structure of an $R$-module. An element $f\in I(X,R)$, such that for any interval $[x,y]\subseteq X$ there exists only a finite number of $[u,v]\subseteq[x,y]$ with $u<v$ and $f(u,v)\ne 0$, is called a {\it finitary series}~\cite{Khripchenko-Novikov09}. The subset $FI(X,R)$ of finitary series is clearly an $R$-submodule of $I(X,R)$. Moreover, for any pair of elements $f,g\in I(X,R)$, at least one of which belongs to $FI(X,R)$, one can define the product $fg$ as the convolution
\begin{align}\label{convolution}
(fg)(x,y)&=\sum_{x\le z\le y}f(x,z)g(z,y),
\end{align}
so that $FI(X,R)$ becomes an $R$-algebra, called the {\it finitary incidence algebra} of $X$ over $R$, and $I(X,R)$ a bimodule over $FI(X,R)$ (see~\cite[Theorem 1]{Khripchenko-Novikov09}). The identity element $\delta$ of $FI(X,R)$ is the function $\delta(x,y)=\delta_{xy}$ for $x\le y$, where $\delta_{xy}\in \{0,1\}$ is the Kronecker delta. Observe that, when $X$ is locally finite, $FI(X,R)=I(X,R)$ is the (classical) incidence algebra~\cite{SpDo}.

An element $f\in FI(X,R)$ is said to be {\it diagonal}, if $f(x,y)=0$ for $x\ne y$. Diagonal elements form a commutative subalgebra of $FI(X,R)$, which we denote by $D(X,R)$. We shall also work with $f\in I(X,R)$ satisfying $f(x,y)=0$ for $x=y$. Such elements form an $FI(X,R)$-submodule of $I(X,R)$ denoted by $Z(X,R)$. Clearly, each $f\in I(X,R)$ can be uniquely written as $f=f_D+f_Z$ with $f_D\in D(X,R)$ and $f_Z\in Z(X,R)$, so $I(X,R)=D(X,R)\oplus Z(X,R)$ as a module over $R$. Consequently, $FZ(X,R):=Z(X,R)\cap FI(X,R)$ is an ideal of $FI(X,R)$ and $FI(X,R)=D(X,R)\oplus FZ(X,R)$ as an $R$-module.

For each pair $x\le y$ define $e_{xy}\in FI(X,R)$ by
\begin{align}\label{e_xy}
e_{xy}(u,v)=
 \begin{cases}
  1, & \mbox{if $u=x$ and $v=y$},\\
  0, & \mbox{otherwise}.
 \end{cases}
\end{align}
Then $e_{xy}e_{uv}=\delta_{yu}e_{xv}$ by the definition of convolution. In particular, the elements $e_x:=e_{xx}\in D(X,R)$, $x\in X$, are pairwise orthogonal idempotents of $FI(X,R)$. Observe that for any $f\in I(X,R)$
\begin{align}\label{e_xfe_y}
 e_xfe_y=\begin{cases}
          f(x,y)e_{xy}, & \mbox{ if }x\le y,\\
          0, & \mbox{ otherwise}.
         \end{cases}
\end{align}
Consequently,
\begin{align}\label{f=g<=>e_xfe_y+e_yfe_x=e_xge_y+e_yge_x}
 f=g&\iff \forall x\le y:\ e_xfe_y=e_xge_y\notag\\
 &\iff \begin{cases}
        \forall x<y:& e_xfe_y+e_yfe_x=e_xge_y+e_yge_x,\\
        \forall x:& e_xfe_x=e_xge_x.
       \end{cases}
\end{align}
The subalgebra of $FI(X,R)$ generated by the functions $e_{xy}$ will be denoted by $\tilde I(X,R)$. As an $R$-module, it admits the decomposition $\tilde I(X,R)=\tilde D(X,R)\oplus \tilde Z(X,R)$, where $\tilde D(X,R)=\tilde I(X,R)\cap D(X,R)$ is a subalgebra of $\tilde I(X,R)$ and $\tilde Z(X,R)=\tilde I(X,R)\cap Z(X,R)$ is an ideal of $\tilde I(X,R)$.

Given $Y\subseteq X$, we introduce the notation $e_Y$ for the diagonal idempotent defined by
$$
e_Y(u,v)=\begin{cases}
          1, & \mbox{ if }u=v\in Y,\\
          0, & \mbox{ otherwise}.
         \end{cases}
$$
In particular, $e_x=e_{\{x\}}$. Note that $e_Ye_Z=e_{Y\cap Z}$, so $e_xe_Y=e_x$ for $x\in Y$, and $e_xe_Y=0$ otherwise.

\section{Jordan isomorphisms of \texorpdfstring{$FI(X,R)$}{FI(X,R)}}\label{jiso-FI}

In all what follows we assume that $(X,\le)$ is an arbitrary (non-necessarily locally finite) partially ordered set, $R$ is a commutative $2$-torsionfree unital ring and $A$ is an associative $R$-algebra. 

Let $\vf$ be a Jordan homomorphism from $FI(X,R)$ to $A$. Then its restriction to $\tilde I(X,R)$ is a Jordan homomorphism $\tilde I(X,R)\to A$. Following the proof of Theorem~2.1 from~\cite{Akkurts-Barker}, one sees that the $R$-linear maps
\begin{align}
 \psi(e_{xy})&=\vf(e_x)\vf(e_{xy})\vf(e_y),\label{psi(e_xy)}\\
 \0(e_{xy})&=\vf(e_y)\vf(e_{xy})\vf(e_x)\label{theta(e_xy)}
\end{align}
are, respectively, a homomorphism and an anti-homomorphism $\tilde I(X,R)\to A$. Moreover, $\vf|_{\tilde I(X,R)}$ is the near-sum of $\psi$ and $\0$ with respect to the subalgebra $\tilde D(X,R)$ and the ideal $\tilde Z(X,R)$ of $\tilde I(X,R)$.

When $\vf$ is bijective, our aim is to show that $\psi$ and $\0$ extend to a homomorphism $\tl\psi$ and an anti-homomorphism $\tl\0$ from $FI(X,R)$ to $A$ in such a way that $\vf$ is the near-sum $\tl\psi+\tl\0$ with respect to $D(X,R)$ and $FZ(X,R)$.

\subsection{The restriction of \texorpdfstring{$\vf$}{phi} to \texorpdfstring{$D(X,R)$}{D(X,R)}}\label{phi-on-D(XR)}
\begin{lem}\label{vf(e_xfe_y+e_yfe_x)}
Let $\vf:FI(X,R)\to A$ be a Jordan homomorphism. Then for any $f\in FI(X,R)$ one has
\begin{align}
\forall x<y:\ f(x,y)\vf(e_{xy})&=\vf(e_x)\vf(f)\vf(e_y)+\vf(e_y)\vf(f)\vf(e_x),\label{f(xy)vf(e_xy)=vf(e_x)vf(f)vf(e_y)+vf(e_y)vf(f)vf(e_x)}\\
\forall x:\ f(x,x)\vf(e_x)&=\vf(e_x)\vf(f)\vf(e_x).\label{f(xx)vf(e_x)=vf(e_x)vf(f)vf(e_x)}
\end{align}
\end{lem}
\begin{proof}
Observe from \cref{e_xfe_y} that $f(x,y)e_{xy}=e_xfe_y+e_yfe_x$ for $x<y$. Therefore, \cref{f(xy)vf(e_xy)=vf(e_x)vf(f)vf(e_y)+vf(e_y)vf(f)vf(e_x)} follows from \cref{vf(abc+cba)}. Similarly \cref{f(xx)vf(e_x)=vf(e_x)vf(f)vf(e_x)} is explained by \cref{vf(aba),e_xfe_y}.
\end{proof}

\begin{lem}\label{a=b-in-A}
 Let $\vf:FI(X,R)\to A$ be a Jordan isomorphism. Given $a,b\in A$, one has
 $$
  a=b\iff\begin{cases}
         \forall x<y:& \vf(e_x)a\vf(e_y)+\vf(e_y)a\vf(e_x)=\vf(e_x)b\vf(e_y)+\vf(e_y)b\vf(e_x),\\
         \forall x:& \vf(e_x)a\vf(e_x)=\vf(e_x)b\vf(e_x).
        \end{cases}
 $$
\end{lem}
\begin{proof}
 The result follows from \cref{vf(abc+cba),vf(aba),f=g<=>e_xfe_y+e_yfe_x=e_xge_y+e_yge_x} and the bijectivity of $\vf$.
\end{proof}

\begin{prop}\label{vf|_D-homo}
Let $\vf:FI(X,R)\to A$ be a Jordan isomorphism. Then $\vf|_{D(X,R)}$ is a homomorphism (and an anti-homomorphism at the same time).
\end{prop}
\begin{proof}
We shall use \cref{a=b-in-A} to prove that $\vf(fg)=\vf(f)\vf(g)$ for all $f,g\in D(X,R)$. Let $x<y$ be arbitrary elements of $X$. Since the idempotents $e_x,e_y$ belong to $D(X,R)$, and $D(X,R)$ is commutative, it follows from \cref{vf(e)vf(a)=vf(a)vf(e),vf(e)vf(a)=0} that
\begin{align*}
  \vf(e_x)\vf(fg)\vf(e_y)&=\vf(e_y)\vf(fg)\vf(e_x)=\vf(fg)\vf(e_x)\vf(e_y)=0,\\
  \vf(e_x)\vf(f)\vf(g)\vf(e_y)&=\vf(e_y)\vf(f)\vf(g)\vf(e_x)=\vf(f)\vf(e_x)\vf(e_y)\vf(g)=0.
\end{align*}
Now by \cref{f(xx)vf(e_x)=vf(e_x)vf(f)vf(e_x)} we have $\vf(e_x)\vf(fg)\vf(e_x)=(fg)(x,x)\vf(e_x)$, and taking into account \cref{vf(a^2),vf(e)vf(a)=vf(a)vf(e)}:
\begin{align*}
 \vf(e_x)\vf(f)\vf(g)\vf(e_x)&=\vf(e_x)\vf(f)\vf(e_x)\cdot\vf(e_x)\vf(g)\vf(e_x)\\
 &=f(x,x)\vf(e_x)g(x,x)\vf(e_x)=f(x,x)g(x,x)\vf(e_x).
\end{align*}
As $(fg)(x,x)=f(x,x)g(x,x)$, the result follows.
\end{proof}

\subsection{An extension of \texorpdfstring{$\psi$}{psi} to \texorpdfstring{$FI(X,R)$}{FI(X,R)}}\label{ext-of-psi}

\begin{lem}\label{f(xy)psi(e_xy)-and-f(xy)theta(e_xy)}
Let $\vf:FI(X,R)\to A$ be a Jordan homomorphism and $\psi$ be given by \cref{psi(e_xy)}. Then for all $f\in FI(X,R)$ and $x\le y$:
\begin{align}
 \vf(e_x)\vf(f)\vf(e_y)&=f(x,y)\psi(e_{xy}).\label{vf(e_x)vf(f)vf(e_y)}
\end{align}
\end{lem}
\begin{proof}
If $x=y$, then $\psi(e_x)=\vf(e_x)$, so \cref{vf(e_x)vf(f)vf(e_y)} is \cref{f(xx)vf(e_x)=vf(e_x)vf(f)vf(e_x)}. If $x<y$, then multiply \cref{f(xy)vf(e_xy)=vf(e_x)vf(f)vf(e_y)+vf(e_y)vf(f)vf(e_x)} by $\vf(e_y)$ on the right and by $\vf(e_x)$ on the left to get 
$$
f(x,y)\vf(e_x)\vf(e_{xy})\vf(e_y)=\vf(e_x)\vf(f)\vf(e_y).
$$
Now, \cref{vf(e_x)vf(f)vf(e_y)} follows from \cref{psi(e_xy)}. 
\end{proof}

\begin{prop}\label{identities-for-psi}
Let $\vf: FI(X,R)\to A$ be a Jordan isomorphism and $\psi$ be given by \cref{psi(e_xy)}. Then there exists an $R$-linear extension $\tl\psi$ of $\psi$ to the whole $FI(X,R)$. Moreover, for any $f\in FZ(X,R)$ one has
 \begin{align}
  \forall x<y:\ \vf(e_x)\tl\psi(f)\vf(e_y)&=\vf(e_x)\vf(f)\vf(e_y),\label{vf(e_x)psi(f)vf(e_y)}\\
  \forall x<y:\ \vf(e_y)\tl\psi(f)\vf(e_x)&=0,\label{vf(e_y)psi(f)vf(e_x)}\\
  \forall x:\ \vf(e_x)\tl\psi(f)\vf(e_x)&=0.\label{vf(e_x)psi(f)vf(e_x)}
 \end{align}
\end{prop}
\begin{proof}
Given $f\in FZ(X,R)$ and $x\le y$, set
\begin{align}\label{a_xy}
  a_{xy}=\vf(e_x)\vf(f)\vf(e_y)\in A
\end{align}
and consider $g\in I(X,R)$ defined by 
\begin{align}\label{g(xy)}
g(x,y)=\vf\m(a_{xy})(x,y).
\end{align}
Observe that $g\in Z(X,R)$, as $a_{xx}=0$ by \cref{f(xx)vf(e_x)=vf(e_x)vf(f)vf(e_x)}. We claim that $g\in FZ(X,R)$. Indeed, suppose that $g(u,v)\ne 0$ for an infinite number of $[u,v]\subseteq[x,y]$ with $u<v$. In view of \cref{g(xy)} one has $\vf\m(a_{uv})\ne 0$ and thus $a_{uv}\ne 0$. Then  $f(u,v)\ne 0$ thanks to \cref{a_xy,vf(e_x)vf(f)vf(e_y)}, which contradicts the fact that $f$ is a finitary series. We now define
\begin{align}\label{psi(f)=vf(g)}
 \tl\psi(f)=\vf(g).
\end{align}
In the general situation, when $f\in FI(X,R)$, write $f=f_D+f_Z$ and thus set $\tl\psi(f)=\vf(f_D)+\tl\psi(f_Z)$.

To prove that $\tl\psi$ is linear, consider $f,g\in FI(X,R)$ and $\alpha\in R$. Then
$$
\tl\psi(\alpha f+g)=\varphi((\alpha f+g)_D)+\tl\psi((\alpha f+g)_Z)=\varphi(\alpha f_D+g_D) + \tl\psi(\alpha f_Z+g_Z).
$$
Since $\varphi$ is linear, it suffices to show that $\tl\psi(\alpha f_1+f_2)=\alpha \tl\psi(f_1)+\tl\psi(f_2)$ for $f_1,f_2\in FZ(X,R)$. As above, we set
$$
\begin{cases} 
a'_{xy}=\varphi(e_x)\varphi(f_1)\varphi(e_y) \\ 
a''_{xy}=\varphi(e_x)\varphi(f_2)\varphi(e_y)
\end{cases}\ \ 
\text{and} \ \ \
\begin{cases} 
g_1(x,y)=\varphi^{-1}(a'_{xy})(x,y) \\ 
g_2(x,y)=\varphi^{-1}(a''_{xy})(x,y)
\end{cases}\red{,}
$$
so that $\tl\psi(f_1)=\varphi(g_1)$ and $\tl\psi(f_2)=\varphi(g_2)$. Now,
$$ 
\alpha \tl\psi(f_1)+\tl\psi(f_2)=\alpha \varphi(g_1)+\varphi(g_2)=\varphi(\alpha g_1+g_2)=\tl\psi(\alpha f_1+f_2),
$$
where the last equality follows from
$$
(\alpha g_1+g_2)(x,y)=\varphi^{-1}(\alpha a'_{xy}+a''_{xy})(x,y) \ \ \text{and} \ \ \alpha a'_{xy}+a''_{xy}=\varphi(e_x)\varphi(\alpha f_1+f_2)\varphi(e_y).
$$

Now, we prove \cref{vf(e_x)psi(f)vf(e_y),vf(e_y)psi(f)vf(e_x),vf(e_x)psi(f)vf(e_x)}. Since $\vf(e_x)\tl\psi(f)\vf(e_x)=\vf(e_x)\vf(g)\vf(e_x)$ by \cref{psi(f)=vf(g)}, equality \cref{vf(e_x)psi(f)vf(e_x)} follows from \cref{f(xx)vf(e_x)=vf(e_x)vf(f)vf(e_x)} and the fact that $g\in FZ(X,R)$.
To prove \cref{vf(e_x)psi(f)vf(e_y),vf(e_y)psi(f)vf(e_x)}, observe from \cref{vf(a^2),vf(e)vf(a)=0} that for all $x<y$
\begin{align}
 \vf(e_x)a_{xy}\vf(e_y)&=a_{xy},\label{vf(e_x)a_xyvf(e_y)}\\
 \vf(e_y)a_{xy}\vf(e_x)&=0.\label{vf(e_y)a_xyvf(e_x)}
\end{align}
Since $\vf\m$ is a Jordan isomorphism,
\begin{align}
 \vf\m(a_{xy})&=\vf\m(\vf(e_x)a_{xy}\vf(e_y)+\vf(e_y)a_{xy}\vf(e_x))\notag\\
 &=e_x\vf\m(a_{xy})e_y+e_y\vf\m(a_{xy})e_x=\vf\m(a_{xy})(x,y)e_{xy}.\label{vf^(-1)(a_xy)}
\end{align}
Hence, by \cref{psi(f)=vf(g),vf^(-1)(a_xy)}
\begin{align}
 \vf(e_x)\tl\psi(f)\vf(e_y)+\vf(e_y)\tl\psi(f)\vf(e_x)&=\vf(e_x)\vf(g)\vf(e_y)+\vf(e_y)\vf(g)\vf(e_x)\notag\\
 &=\vf(e_xge_y+e_yge_x)=\vf(\vf\m(a_{xy})(x,y)e_{xy})\notag\\
 &=\vf(\vf\m(a_{xy}))=a_{xy}.\label{vf(e_x)psi(f)vf(e_y)+vf(e_y)psi(f)vf(e_x)=a_xy}
\end{align}
Multiplying this by $\vf(e_x)$ on the left and by $\vf(e_y)$ on the right and using \cref{vf(e_x)a_xyvf(e_y),a_xy} we get \cref{vf(e_x)psi(f)vf(e_y)}. Similarly, \cref{vf(e_y)a_xyvf(e_x)} and the multiplication of \cref{vf(e_x)psi(f)vf(e_y)+vf(e_y)psi(f)vf(e_x)=a_xy} by $\vf(e_y)$ on the left and by $\vf(e_x)$ on the right give \cref{vf(e_y)psi(f)vf(e_x)}.

We now show that $\tl\psi$ is an extension of $\psi$, i.\,e. it satisfies \cref{psi(e_xy)}. This is clearly true, when $x=y$, since $\tl\psi$ coincides with $\vf$ on $D(X,R)$. Let $x<y$, so that $e_{xy}\in FZ(X,R)$. We shall use \cref{a=b-in-A}. By \cref{vf(e_x)psi(f)vf(e_x)}, for any $u$,
$$
\vf(e_u)\tl\psi(e_{xy})\vf(e_u)=0=\vf(e_u)\vf(e_x)\vf(e_{xy})\vf(e_y)\vf(e_u),
$$
since either $u\ne x$ or $u\ne y$. Furthermore, by \cref{vf(e_x)psi(f)vf(e_y),vf(e_y)psi(f)vf(e_x)} for any $u<v$,
\begin{align*}
 \vf(e_u)\tl\psi(e_{xy})\vf(e_v)+\vf(e_v)\tl\psi(e_{xy})\vf(e_u)&=\vf(e_u)\vf(e_{xy})\vf(e_v).
\end{align*}
If $\{u,v\}\ne\{x,y\}$, then the latter is zero, since $e_ue_{xy}=e_{xy}e_u=0$ or $e_ve_{xy}=e_{xy}e_v=0$. But
$$
\vf(e_u)\vf(e_x)\vf(e_{xy})\vf(e_y)\vf(e_v)+\vf(e_v)\vf(e_x)\vf(e_{xy})\vf(e_y)\vf(e_u)
$$
is also zero in this case. Since $u<v$, it remains only one case to be checked: $u=x$ and $v=y$. In this situation,
\begin{align*}
 \vf(e_u)\vf(e_{xy})\vf(e_v)&=\vf(e_x)\vf(e_{xy})\vf(e_y)\\
 &=\vf(e_x)\vf(e_x)\vf(e_{xy})\vf(e_y)\vf(e_y)+\vf(e_y)\vf(e_x)\vf(e_{xy})\vf(e_y)\vf(e_x).
\end{align*}
\end{proof}

The proof that $\tl\psi$ respects multiplication will be divided into several lemmas.

\begin{lem} \label{f-in-D-and-g-in-Z}
Let $\vf:FI(X,R)\to A$ be a Jordan isomorphism and $\tl\psi:FI(X,R)\to A$ as defined in \cref{identities-for-psi}. If $f\in D(X,R)$ and  $g\in FZ(X,R)$, then $\tl\psi(fg)=\tl\psi(f)\tl\psi(g)$.
\end{lem}
\begin{proof}
To prove the equality, we shall use \cref{a=b-in-A}. Let $x\le y$. If $x<y$, then by \cref{vf(e_x)psi(f)vf(e_y),vf(e_x)vf(f)vf(e_y),vf(e_y)psi(f)vf(e_x)}
\begin{align}\label{vf(e_x)psi(fg)vf(e_y)+vf(e_y)psi(fg)vf(e_x)}
\varphi(e_x)\tl\psi(fg)\varphi(e_y)+\varphi(e_y)\tl\psi(fg)\varphi(e_x)= (fg)(x,y)\tl\psi(e_{xy})= f(x,x)g(x,y)\tl\psi(e_{xy}).
\end{align}
Now, we need to compute 
\begin{align}\label{vf(e_x)psi(f)psi(g)vf(e_y)+vf(e_y)psi(f)psi(g)vf(e_x)}
\varphi(e_x)\tl\psi(f)\tl\psi(g)\vf(e_y)+\vf(e_y)\tl\psi(f)\tl\psi(g)\vf(e_x).
\end{align}
Since $e_x, f\in D(X,R)$, $e_xf=f(x,x)e_x$ and $\tl\psi|_{D(X,R)}=\vf|_{D(X,R)}$ is a homomorphism, we have by \cref{vf(e_x)psi(f)vf(e_y),vf(e_x)vf(f)vf(e_y)}
\begin{align*}
\varphi(e_x)\tl\psi(f)\tl\psi(g)\vf(e_y)&=\varphi(e_x)\vf(f)\tl\psi(g)\vf(e_y)=f(x,x)\varphi(e_x)\tl\psi(g)\vf(e_y)\\
&=f(x,x)g(x,y)\tl\psi(e_{xy}).
\end{align*}
Similarly by \cref{vf(e_y)psi(f)vf(e_x)}
\begin{align*}
\varphi(e_y)\tl\psi(f)\tl\psi(g)\vf(e_x)=f(y,y)\varphi(e_y)\tl\psi(g)\vf(e_x)=0,
\end{align*}
so \cref{vf(e_x)psi(f)psi(g)vf(e_y)+vf(e_y)psi(f)psi(g)vf(e_x)} coincides with the last term of \cref{vf(e_x)psi(fg)vf(e_y)+vf(e_y)psi(fg)vf(e_x)}.

If $x=y$, then $\vf(e_x)\tl\psi(fg)\vf(e_x)=0$ by \cref{vf(e_x)psi(f)vf(e_x)},
because $fg\in FZ(X,R)$. By the same reason
\begin{align*}
\varphi(e_x)\tl\psi(f)\tl\psi(g)\vf(e_x)=f(x,x)\vf(e_x)\tl\psi(g)\vf(e_x)=0.
\end{align*}
\end{proof}

\begin{rem}\label{f-in-Z-and-g-in-D}
By a similar computation one proves that if $f\in FZ(X,R)$ and $g\in D(X,R)$, then $\tl\psi(fg)=\tl\psi(f)\tl\psi(g)$.
\end{rem}
It remains to consider the case $f,g\in FZ(X,R)$. To treat it, we shall need two technical lemmas.

\begin{lem}\label{vf(abcde+...)}
Let $\vf:A\to B$ be a Jordan homomorphism. Given $a,b,c,d,e\in A$, one has
\begin{align*}
 \vf(abcde+edabc+cbade+edcba)&=\vf(a)\vf(b)\vf(c)\vf(d)\vf(e)+\vf(e)\vf(d)\vf(a)\vf(b)\vf(c)\\
&\quad+\vf(c)\vf(b)\vf(a)\vf(d)\vf(e)+\vf(e)\vf(d)\vf(c)\vf(b)\vf(a).
\end{align*}
\end{lem}
\begin{proof}
By \cref{vf(abc+cba)},
\begin{align*}
 \vf((abc)de+ed(abc))&=\vf(abc)\vf(d)\vf(e)+\vf(e)\vf(d)\vf(abc),\\
 \vf((cba)de+ed(cba))&=\vf(cba)\vf(d)\vf(e)+\vf(e)\vf(d)\vf(cba).
\end{align*}
It remains to add these equalities and use \cref{vf(abc+cba)} once again.
\end{proof}

\begin{lem}\label{vf(e_x)vf(f)vf(e_X-[xy])vf(g)vf(e_y)=0}
Let $\vf:FI(X,R)\to A$ be a Jordan isomorphism and $\tl\psi:FI(X,R)\to A$ as defined in \cref{identities-for-psi}. Given $f,g\in FZ(X,R)$ and $f',g'\in FZ(X,R)$ such that $\tl\psi(f)=\vf(f')$ and $\tl\psi(g)=\vf(g')$, one has for all $x\le y$ and $W\subseteq X\setminus\{z\in[x,y]\mid f'(x,z)\ne 0\ne g'(z,y)\}$
$$ 
\vf(e_x)\tl\psi(f)\vf(e_W)\tl\psi(g)\vf(e_y)=\vf(e_y)\tl\psi(f)\vf(e_W)\tl\psi(g)\vf(e_x)=0. 
$$
\end{lem}
\begin{proof}
Assume first that $x<y$. Let us apply \cref{vf(abcde+...)} with $a=e_x$, $b=f'$, $c=e_W$, $d=g'$ and $e=e_y$. The products $e_yg'e_xf'e_W$ and $e_yg'e_Wf'e_x$ are zero thanks to \cref{e_xfe_y}. The product $e_xf'e_Wg'e_y$ equals $(f'e_Wg')(x,y)e_{xy}$, which is also zero by \cref{convolution}. Thus,
\begin{align}
\vf(e_Wf'e_xg'e_y)&=\vf(e_x)\vf(f')\vf(e_W)\vf(g')\vf(e_y)\notag\\
 &\quad+\vf(e_y)\vf(g')\vf(e_x)\vf(f')\vf(e_W)\notag\\
 &\quad+\vf(e_W)\vf(f')\vf(e_x)\vf(g')\vf(e_y)\notag\\
 &\quad+\vf(e_y)\vf(g')\vf(e_W)\vf(f')\vf(e_x).\label{vf(e_X-minus[xy]fe_xge_y)}
\end{align}
Suppose first that $x\in W$. Multiplying \cref{vf(e_X-minus[xy]fe_xge_y)} by $\vf(e_x)$ on the left and using \cref{vf(e)vf(a)=0}, we get zero on the left-hand side of \cref{vf(e_X-minus[xy]fe_xge_y)}, as $e_xe_Wf'e_x=f'(x,x)e_x=0$ and $e_ye_x=0$. The second and fourth terms of the right-hand side of \cref{vf(e_X-minus[xy]fe_xge_y)} are also zero by \cref{vf(e)vf(a)=0}, the third one is zero by \cref{vf(e_x)vf(f)vf(e_y),vf|_D-homo}, while the first one does not change by \cref{vf(a^2)}. Similarly the multiplication of \cref{vf(e_X-minus[xy]fe_xge_y)} by $\vf(e_x)$ on the right gives $\vf(e_y)\vf(g')\vf(e_W)\vf(f')\vf(e_x)=0$. It remains to replace $\vf(f')$ by $\tl\psi(f)$ and $\vf(g')$ by $\tl\psi(g)$. The same argument works, when $x\not\in W$, since $e_xe_W=e_We_x=0$ in this case.

Now let $x=y$. Observe that $W$ is an arbitrary subset of $X$ in this case. We first prove the following:
\begin{align}\label{vf(f)-to-vf(f_>x)}
\vf(e_x)\tl\psi(f)\vf(e_W)=\vf(e_x)\tl\psi(f_{>x})\vf(e_W),
\end{align}
where
\begin{align}\label{f_>x}
 f_{>x}(u,v)=\begin{cases}
             f(u,v), & \mbox{ if $u=x$ and $v>x$},\\
             0, & \mbox{ otherwise}.
            \end{cases}
\end{align}
To this end we use \cref{a=b-in-A}. Thanks to \cref{vf(e)vf(a)=0,vf(e_x)psi(f)vf(e_x)}, for any $u$ both sides of \cref{vf(f)-to-vf(f_>x)} become zero after multiplication by $\vf(e_u)$ on the left and on the right. Now take $u<v$ and consider
\begin{align}\label{vf(e_u)vf(f)vf(e_v)-to-vf(e_u)vf(f_>x)vf(e_v)}
\vf(e_u)\vf(e_x)\tl\psi(f)\vf(e_W)\vf(e_v)+\vf(e_v)\vf(e_x)\tl\psi(f)\vf(e_W)\vf(e_u).
\end{align}
If $x\not\in\{u,v\}$ or $\{u,v\}\cap W=\emptyset$, then \cref{vf(e_u)vf(f)vf(e_v)-to-vf(e_u)vf(f_>x)vf(e_v)} is zero for any $f$ by \cref{vf(e)vf(a)=0}. Let $x=u<v\in W$. Then the sum \cref{vf(e_u)vf(f)vf(e_v)-to-vf(e_u)vf(f_>x)vf(e_v)} is $\vf(e_x)\tl\psi(f)\vf(e_v)$ by \cref{vf(a^2),vf(e)vf(a)=0}, which by \cref{vf(e_x)psi(f)vf(e_y),vf(e_x)vf(f)vf(e_y)} equals
$$
 f(x,v)\tl\psi(e_{xv})=f_{>x}(x,v)\tl\psi(e_{xv})=\vf(e_x)\tl\psi(f_{>x})\vf(e_v),
$$
so $f$ can be replaced by $f_{>x}$ in \cref{vf(e_u)vf(f)vf(e_v)-to-vf(e_u)vf(f_>x)vf(e_v)}. If $W\ni u<v=x$, then \cref{vf(e_u)vf(f)vf(e_v)-to-vf(e_u)vf(f_>x)vf(e_v)} becomes $\vf(e_x)\tl\psi(f)\vf(e_u)$, which is zero for any $f\in FZ(X,R)$ by \cref{vf(e_y)psi(f)vf(e_x)}. If $W\ni x=u<v\not\in W$ or $W\not\ni u<v=x\in W$, then \cref{vf(e_u)vf(f)vf(e_v)-to-vf(e_u)vf(f_>x)vf(e_v)} is again zero for any $f$.

Similarly we get
\begin{align}\label{vf(f)-to-vf(f_<x)}
\vf(e_W)\tl\psi(f)\vf(e_x)=\vf(e_W)\tl\psi(f_{<x})\vf(e_x),
\end{align}
where
\begin{align}\label{f_<x}
 f_{<x}(u,v)=\begin{cases}
             f(u,v), & \mbox{ if $u<x$ and $v=x$},\\
             0, & \mbox{ otherwise}.
            \end{cases}
\end{align}
Therefore, using \cref{vf(a^2)} we deduce from \cref{vf(f)-to-vf(f_>x),vf(f)-to-vf(f_<x)}
\begin{align}\label{vf(e_x)vf(f)vf(e_X-x)vf(g)vf(e_x)=vf(e_x)vf(f_>x)vf(e_X-x)vf(g_<x)vf(e_x)}
 \vf(e_x)\tl\psi(f)\vf(e_W)\tl\psi(g)\vf(e_x)=\vf(e_x)\tl\psi(f_{>x})\vf(e_W)\tl\psi(g_{<x})\vf(e_x).
\end{align}

Now write $\tl\psi(f_{>x})=\vf(f'')$, $\tl\psi(g_{<x})=\vf(g'')$ for some $f'',g''\in FZ(X,R)$ and apply \cref{vf(abcde+...)} with $a=e_x$, $b=f''$, $c=e_W$, $d=g''$ and $e=e_x$. The products $e_xg''e_xf''e_W$ and $e_Wf''e_xg''e_x$ are zero by \cref{e_xfe_y} and the fact that $g''\in FZ(X,R)$. Similarly $e_xf''e_Wg''e_x=e_xg''e_Wf''e_x=0$, since $f''e_Wg'',g''e_Wf''\in FZ(X,R)$. Hence,
\begin{align}
 0&=\vf(e_x)\vf(f'')\vf(e_W)\vf(g'')\vf(e_x)
 +\vf(e_x)\vf(g'')\vf(e_x)\vf(f'')\vf(e_W)\notag\\
 &\quad+\vf(e_W)\vf(f'')\vf(e_x)\vf(g'')\vf(e_x)+\vf(e_x)\vf(g'')\vf(e_W)\vf(f'')\vf(e_x).\label{vf(e_x)vf(f_>x)vf(e_X-x)vf(g_<x)vf(e_x)+...=0}
\end{align}
By \cref{vf(e_x)vf(f)vf(e_y)} the second and third summands of \cref{vf(e_x)vf(f_>x)vf(e_X-x)vf(g_<x)vf(e_x)+...=0} are zero, because $g''\in FZ(X,R)$. Thus,
\begin{align}\label{psi><}
\vf(e_x)\tl\psi(f_{>x})\vf(e_W)\tl\psi(g_{<x})\vf(e_x)+\vf(e_x)\tl\psi(g_{<x})\vf(e_W)\tl\psi(f_{>x})\vf(e_x)=0.
\end{align}
But
$$ \vf(e_x)\tl\psi(g_{<x})\vf(e_W)\tl\psi(f_{>x})\vf(e_x)=\vf(e_x)\tl\psi((g_{<x})_{>x})\vf(e_W)\tl\psi((f_{>x})_{<x})\vf(e_x)
$$
thanks to \cref{vf(e_x)vf(f)vf(e_X-x)vf(g)vf(e_x)=vf(e_x)vf(f_>x)vf(e_X-x)vf(g_<x)vf(e_x)}, which is zero, as $(g_{<x})_{>x}=(f_{>x})_{<x}=0$ by \cref{f_<x,f_>x}. So,
$$
\vf(e_x)\tl\psi(f_{>x})\vf(e_W)\tl\psi(g_{<x})\vf(e_x)=0,
$$
which in view of \cref{vf(e_x)vf(f)vf(e_X-x)vf(g)vf(e_x)=vf(e_x)vf(f_>x)vf(e_X-x)vf(g_<x)vf(e_x)} gives the desired equality.
\end{proof}

\begin{lem}\label{f-in-Z-and-g-in-Z}
Let $\vf:FI(X,R)\to R$ be a Jordan isomorphism and $\tl\psi:FI(X,R)\to A$ as defined in \cref{identities-for-psi}. If $f,g\in FZ(X,R)$, then $\tl\psi(fg)=\tl\psi(f)\tl\psi(g)$.
\end{lem}
\begin{proof}
As in the proof of \cref{f-in-D-and-g-in-Z} for $x<y$ we have
$$
\varphi(e_x)\tl\psi(fg)\varphi(e_y)+\varphi(e_y)\tl\psi(fg)\varphi(e_x)= (fg)(x,y)\tl\psi(e_{xy}),
$$
the latter being
\begin{align*}
\sum_{z\in Z}f(x,z)\tl\psi(e_{xz})g(z,y)\tl\psi(e_{zy})
&= \sum_{z\in Z}\vf(e_x)\vf(f)\vf(e_z)\vf(g)\vf(e_y)\\
&=\vf(e_x)\vf(f)\vf(e_Z)\vf(g)\vf(e_y),
\end{align*}
where $Z$ is any finite subset of $[x,y]$ which contains $\{z\in[x,y]\mid f(x,z)\ne 0\ne g(z,y)\}$. In particular, we may take
$$
Z=\{z\in[x,y]\mid f(x,z)\ne 0\ne g(z,y)\}\cup\{z\in[x,y]\mid f'(x,z)\ne 0\ne g'(z,y)\}
$$
with $f',g'\in FZ(X,R)$ being such that $\tl\psi(f)=\vf(f')$ and $\tl\psi(g)=\vf(g')$.

On the other hand,
\begin{align}
\vf(e_x)\tl\psi(f)\tl\psi(g)\vf(e_y)&= \vf(e_x)\tl\psi(f)\vf(e_Z)\tl\psi(g)\vf(e_y)\notag\\
&\quad+\vf(e_x)\tl\psi(f)\vf(e_{X\setminus Z})\tl\psi(g)\vf(e_y),\label{vf(e_x)psi(f)psi(g)vf(e_y)}\\
\vf(e_y)\tl\psi(f)\tl\psi(g)\vf(e_x)&= \vf(e_y)\tl\psi(f)\vf(e_Z)\tl\psi(g)\vf(e_x)\notag\\
&\quad+\vf(e_y)\tl\psi(f)\vf(e_{X\setminus Z})\tl\psi(g)\vf(e_x).\label{vf(e_y)psi(f)psi(g)vf(e_x)}
\end{align}
By \cref{vf(e_x)vf(f)vf(e_X-[xy])vf(g)vf(e_y)=0} the second summands of the right-hand sides of \cref{vf(e_x)psi(f)psi(g)vf(e_y),vf(e_y)psi(f)psi(g)vf(e_x)} are zero. Moreover,
$$
\vf(e_y)\tl\psi(f)\vf(e_Z)\tl\psi(g)\vf(e_x)=\sum_{z\in Z} \vf(e_y)\tl\psi(f)\vf(e_z)\tl\psi(g)\vf(e_x)=0,
$$
because $\vf(e_z)\tl\psi(g)\vf(e_x)=0$ for all $x\leq z$ by \cref{vf(e_y)psi(f)vf(e_x),vf(e_x)psi(f)vf(e_x)}. Finally,
by \cref{vf(e_x)psi(f)vf(e_y)}
\begin{align*}
\varphi(e_x)\tl\psi(f)\vf(e_Z)\tl\psi(g)\vf(e_y) & = \sum_{z\in Z}\varphi(e_x)\vf(f)\vf(e_z)\vf(g)\vf(e_y)\\
& = \varphi(e_x)\vf(f)\vf(e_Z)\vf(g)\vf(e_y),
\end{align*}
proving that
$$
\vf(e_x)\tl\psi(fg)\vf(e_y)+\vf(e_y)\tl\psi(fg)\vf(e_x)= \vf(e_x)\tl\psi(f)\tl\psi(g)\vf(e_y)+\vf(e_y)\tl\psi(f)\tl\psi(g)\vf(e_x).
$$

Now using the fact that $fg\in FZ(X,R)$ and  \cref{vf(e_x)psi(f)vf(e_x)}, we get
$$
\vf(e_x)\tl\psi(fg)\vf(e_x)=0.
$$
On the other hand, 
\begin{align*}
\varphi(e_x)\tl\psi(f)\tl\psi(g)\vf(e_x)=\varphi(e_x)\tl\psi(f)\varphi(\delta)\tl\psi(g)\vf(e_x)=\varphi(e_x)\tl\psi(f)\vf(e_X)\tl\psi(g)\vf(e_x),
\end{align*}
which is zero thanks to \cref{vf(e_x)vf(f)vf(e_X-[xy])vf(g)vf(e_y)=0}. The result now follows from \cref{a=b-in-A}.
\end{proof}

\begin{prop} \label{psi-homo}
Let $\vf:FI(X,R)\to R$ be a Jordan isomorphism and $\tl\psi:FI(X,R)\to A$ as defined in \cref{identities-for-psi}. Then $\tl\psi$ is a homomorphism.
\end{prop}
\begin{proof}
Consider $f,g \in FI(X,R)$ and write $f=f_D+f_Z$ and $g=g_D+g_Z$. Using the definition of $\tl\psi$, its linearity, \cref{f-in-D-and-g-in-Z,f-in-Z-and-g-in-D,vf|_D-homo,f-in-Z-and-g-in-Z}, we get
\begin{align*}
\tl\psi(fg) & = \tl\psi((f_D+f_Z)(g_D+g_Z))\\ 
 & = \tl\psi(f_Dg_D+(f_Dg_Z+f_Zg_D+f_Zg_Z))\\
 & = \vf(f_Dg_D)+ \tl\psi(f_Dg_Z+f_Zg_D+f_Zg_Z)\\ 
 & = \vf(f_D)\vf(g_D)+\tl\psi(f_Dg_Z)+ \tl\psi(f_Zg_D)+ \tl\psi(f_Zg_Z)\\
 & = \vf(f_D)\vf(g_D)+\tl\psi(f_D)\tl\psi(g_Z)+ \tl\psi(f_Z)\tl\psi(g_D)+ \tl\psi(f_Z)\tl\psi(g_Z)\\
 & = \vf(f_D)\vf(g_D)+\vf(f_D)\tl\psi(g_Z)+ \tl\psi(f_Z)\vf(g_D)+ \tl\psi(f_Z)\tl\psi(g_Z)\\
 & =(\vf(f_D)+\tl\psi(f_Z))(\vf(g_D)+\tl\psi(g_Z))\\ 
 & = \tl\psi(f)\tl\psi(g).
\end{align*}
\end{proof}

\subsection{A decomposition of \texorpdfstring{$\vf$}{phi} into a near-sum}\label{decomp-of-phi}

Similarly to what was done in \cref{identities-for-psi} for $\psi$, one can extend $\0$ defined in \cref{theta(e_xy)} to an anti-homomorphism $FI(X,R)\to A$.

\begin{prop}\label{identities-for-theta}
Let $\vf: FI(X,R)\to A$ be a Jordan isomorphism and $\0$ be given by \cref{theta(e_xy)}. Then there exists an extension $\tl\0$ of $\0$ to an anti-homomorphism $FI(X,R)\to A$. Moreover, for any $f\in FZ(X,R)$ one has
\begin{align}
  \forall x<y:\ \vf(e_y)\tl\0(f)\vf(e_x)&=\vf(e_y)\vf(f)\vf(e_x),\label{vf(e_y)theta(f)vf(e_x)}\\
  \forall x<y:\ \vf(e_x)\tl\0(f)\vf(e_y)&=0,\label{vf(e_x)theta(f)vf(e_y)}\\
  \forall x:\ \vf(e_x)\tl\0(f)\vf(e_x)&=0.\label{vf(e_x)theta(f)vf(e_x)}
\end{align}
\end{prop}
\begin{proof}
Given $f\in FZ(X,R)$ and $x\le y$, set
\begin{align*}
  b_{xy}=\vf(e_y)\vf(f)\vf(e_x)\in A
\end{align*}
and define $h\in I(X,R)$ by 
\begin{align}\label{h(xy)}
h(x,y)=\vf\m(b_{xy})(x,y),
\end{align}
for all $x\leq y$. Observe as in the proof of \cref{identities-for-psi} that $h\in FZ(X,R)$. Define
\begin{align*}
 \tl\0(f)=\vf(h).
\end{align*}
In the general situation, when $f\in FI(X,R)$, write $f=f_D+f_Z$ and thus set $\tl\theta(f)=\vf(f_D)+\tl\theta(f_Z)$.

The proof that $\tl\theta$ is an anti-homomorphism satisfying \cref{vf(e_y)theta(f)vf(e_x),vf(e_x)theta(f)vf(e_y),vf(e_x)theta(f)vf(e_x)} 
is analogous to what was done for $\tl\psi$ (see \cref{identities-for-psi,f-in-D-and-g-in-Z,f-in-Z-and-g-in-D,vf(e_x)vf(f)vf(e_X-[xy])vf(g)vf(e_y)=0,f-in-Z-and-g-in-Z,psi-homo}).
\end{proof}

Finally, we are ready to prove the main result of this paper.

\begin{thrm}\label{near-sum}
Each Jordan isomorphism $\vf:FI(X,R)\to A$ is the near-sum of $\tl\psi$ and $\tl\theta$ with respect to the subalgebra $D(X,R)$ and the ideal $FZ(X,R)$, where $\tl\psi$ and $\tl\theta$ are defined in \cref{identities-for-psi,identities-for-theta}, respectively.
\end{thrm}
\begin{proof}
By definition, $\tl\psi|_{D(X,R)}=\tl\theta|_{D(X,R)}=\vf|_{D(X,R)}$. Now, for $f\in FZ(X,R)$, $\tl\psi(f)+\tl\theta(f)=\vf(g)+\vf(h)=\vf(g+h)$, where $g$ and $h$ are defined in \cref{g(xy),h(xy)}. Since $\vf\m$, being a Jordan isomorphism, satisfies \cref{vf(abc+cba)}, for all $x\leq y$,
\begin{align*}
(g+h)(x,y) & = g(x,y)+h(x,y)=\vf^{-1}(a_{xy})(x,y)+\vf^{-1}(b_{xy})(x,y)\\
 & = [\vf^{-1}(a_{xy})+\vf^{-1}(b_{xy})](x,y)=\vf^{-1}(a_{xy}+b_{xy})(x,y)\\
 & = [\vf^{-1}(\vf(e_{x})\vf(f)\vf({e_y})+\vf(e_y)\vf(f)\vf(e_x))](x,y)\\
 & = [e_xfe_y+e_yfe_x](x,y) = f(x,y).
\end{align*}
Thus $\vf(f)=\tl\psi(f)+\tl\theta(f)$, for all $f\in FZ(X,R)$, i.\,e. $\vf|_{FZ(X,R)}=\tl\psi|_{FZ(X,R)}+\tl\theta|_{FZ(X,R)}$.

It remains to check that the products $\tl\psi(f)\tl\0(f')$ and $\tl\0(f')\tl\psi(f)$ are zero for all $f,f'\in FZ(X,R)$. By symmetry, it is enough to show that $\tl\psi(f)\tl\0(f')=0$. To this end we use \cref{a=b-in-A}.

Let $x<y$. Using the fact that $\tl\psi(f)=\vf(g)$ and $\tl\0(f')=\vf(h')$, where $g$ and $h'$ are defined in \cref{g(xy),h(xy)}, we represent
\begin{align}\label{vf(e_x)vf(g)vf(h')vf(e_y)+vf(e_y)vf(g)vf(h')vf(e_x)}
\vf(e_x)\tl\psi(f)\tl\0(f')\vf(e_y)+\vf(e_y)\tl\psi(f)\tl\0(f')\vf(e_x)
\end{align}
as
\begin{align}
 &\sum_{z\in Z}\vf(e_x)\vf(g)\vf(e_z)\vf(h')\vf(e_y)
 +\vf(e_x)\vf(g)\vf(e_{X\setminus Z})\vf(h')\vf(e_y)\notag\\
 &\quad+\sum_{z\in Z}\vf(e_y)\vf(g)\vf(e_z)\vf(h')\vf(e_x)
 +\vf(e_y)\vf(g)\vf(e_{X\setminus Z})\vf(h')\vf(e_x),\label{vf(e_x)vf(g)vf(h')vf(e_y)+vf(e_y)vf(g)vf(h')vf(e_x)-expanded}
\end{align}
where $Z$ is a finite subset of $[x,y]$ which contains $\{z\in[x,y]\mid f(x,z)\ne 0\ne g(z,y)\}$. As in the proof of \cref{vf(e_x)vf(f)vf(e_X-[xy])vf(g)vf(e_y)=0} the second and fourth summands of \cref{vf(e_x)vf(g)vf(h')vf(e_y)+vf(e_y)vf(g)vf(h')vf(e_x)-expanded} are zero.
Moreover, $\vf(e_z)\vf(h')\vf(e_y)=\vf(e_z)\tl\theta(f')\vf(e_y)=0$ for all $z\in Z$ thanks to \cref{vf(e_x)theta(f)vf(e_y),vf(e_x)theta(f)vf(e_x)}. Similarly, $\vf(e_y)\vf(g)\vf(e_z)=\vf(e_y)\tl\psi(f)\vf(e_z)=0$ for all $z\in Z$
in view of \cref{vf(e_y)psi(f)vf(e_x),vf(e_x)psi(f)vf(e_x)}. Therefore, the first and third summands of \cref{vf(e_x)vf(g)vf(h')vf(e_y)+vf(e_y)vf(g)vf(h')vf(e_x)-expanded} are also zero, yielding that \cref{vf(e_x)vf(g)vf(h')vf(e_y)+vf(e_y)vf(g)vf(h')vf(e_x)} equals $0$. 

Now take $x\in X$ and write 
\begin{align}\label{vf(e_x)psi(f)0(f')vf(e_x)}
\vf(e_x)\tl\psi(f)\tl\0(f')\vf(e_x)=\vf(e_x)\tl\psi(f)\vf(e_X)\tl\0(f')\vf(e_x).
\end{align}
To prove that the right-hand side of \cref{vf(e_x)psi(f)0(f')vf(e_x)} is zero, we observe as in the proof of \cref{vf(e_x)vf(f)vf(e_X-[xy])vf(g)vf(e_y)=0} that
\begin{align}
\vf(e_x)\tl\theta(f')\vf(e_X)&=\vf(e_x)\tl\theta(f'_{<x})\vf(e_X),\label{0(f)-to-0(f_>x)}\\
\vf(e_X)\tl\theta(f')\vf(e_x)&=\vf(e_X)\tl\theta(f'_{>x})\vf(e_x)\label{0(f)-to-0(f_<x)}
\end{align}
(see \cref{vf(f)-to-vf(f_>x),vf(f)-to-vf(f_<x)}). It follows from \cref{vf(f)-to-vf(f_>x),0(f)-to-0(f_<x)} that
\begin{align}\label{psi_>0_>}
\vf(e_x)\tl\psi(f)\vf(e_X)\tl\0(f')\vf(e_x)=\vf(e_x)\tl\psi(f_{>x})\vf(e_X)\tl\0(f'_{>x})\vf(e_x).
\end{align}
Similarly \cref{vf(f)-to-vf(f_<x),0(f)-to-0(f_>x)} yield
\begin{align}\label{0_<psi_<}
\vf(e_x)\tl\theta(f')\vf(e_X)\tl\psi(f)\vf(e_x)=\vf(e_x)\tl\0(f'_{<x})\vf(e_X)\tl\psi(f_{<x})\vf(e_x).
\end{align}
Observe as in \cref{psi><} that
\begin{align}\label{switch-psi-and-0}
\vf(e_x)\tl\psi(f_{>x})\vf(e_X)\tl\theta(f'_{>x})\vf(e_x)+\vf(e_x)\tl\theta(f'_{>x})\vf(e_X)\tl\psi(f_{>x})\vf(e_x)=0.
\end{align}
But $\vf(e_x)\tl\theta(f'_{>x})\vf(e_X)\tl\psi(f_{>x})\vf(e_x)$ equals
$$
\vf(e_x)\tl\theta((f'_{>x})_{<x})\vf(e_X)\tl\psi((f_{>x})_{<x})\vf(e_x)=0
$$
in view of \cref{0_<psi_<,f_<x,f_>x}. Thus, $\vf(e_x)\tl\psi(f)\vf(e_X)\tl\0(f')\vf(e_x)=0$
by \cref{psi_>0_>,switch-psi-and-0}, as desired.
\end{proof}

\section*{Acknowledgments}
The first and second authors were partially supported by Funda\c{c}\~ao Arauc\'aria, Conv\^enio 212/14. The third author thanks the Department of Mathematics of Maring\'a State University for its warm hospitality and financial support.

\bibliography{bibl}{}
\bibliographystyle{acm}

\end{document}